\documentclass[11pt]{article}

\usepackage{amsmath}
\usepackage{amsthm}
\usepackage{makeidx}
\usepackage{latexsym}
\usepackage{graphicx}
\usepackage{amssymb}
\usepackage{hyperref}
\usepackage[latin1]{inputenc}

\newtheorem{teo}{Theorem}

\newtheorem{oss}{Remark}

\theoremstyle{remark}
\newtheorem{es}{\textbf{Example}}

\title{Squaring the magic squares of order 4}
\author{Stefano Barbero, Umberto Cerruti, Nadir Murru\\
Department of Mathematics\\
University of Turin}

\date{}

\begin{document}
\maketitle

\begin{abstract}
In this paper, we present the problem of counting magic squares and we focus on the case of multiplicative magic squares of order 4. We give the exact number of normal multiplicative magic squares of order 4 with an original and complete proof, pointing out the role of the action of the symmetric group. Moreover, we provide a new representation for magic squares of order 4. Such representation allows the construction of magic squares in a very simple way, using essentially only five particular $4\times 4$ matrices .
\end{abstract}

\section{Introduction}
A magic square is defined as an $n\times n$ matrix of integers, where the sum of the numbers in each \emph{line} (i.e., in each row, in each column and in each diagonal) is the same. Magic squares have a very rich history (see, e.g., the beautiful book of Descombes \cite{Des}) and they have many generalizations. There are a lot of contributes to their theory from several different fields (see, e.g., the classical book of Andrews \cite{Andrews}). \\
A very difficult job is counting how many magic squares there are for a given order and a given line sum, but if we relax the requirements and we ask for the so--called semi--magic squares (i.e., magic squares without the condition on the diagonals), the counting is easier. For order 3, Bona \cite{Bona}  has found that the number of semi--magic squares is 
$$\binom{r+4}{4}+\binom{r+3}{4}+\binom{r+2}{4}\ ,$$ 
where $r$ is the line sum (let us observe that in \cite{Bona} semi--magic squares are called magic squares). For a general theory of counting semi--magic squares you can see \cite{Bona2}, chapter 9.\\
The classic magic squares, which are called \emph{normal}, are $n\times n$ matrices whose entries consist of the numbers $0,1,...,n^2-1$ and in each line the sum of the numbers is the same. Their number is known for the orders 3, 4 and 5. The case of order 4 is especially interesting. There are exactly 880 normal magic squares of order 4 (up to symmetries of the square, i.e., in total they are $880\times8=7040$). They were enumerated for the first time in 1693 by Frenicle de Bessy. The Frenicle method was analytically expanded and completed by Bondi and Ollerenshaw \cite{Bondi}. Of course, one can think to different kinds of magic squares, only changing the operation. Then a \emph{multiplicative magic square} is an $n\times n$ matrix of integers, where the product of the numbers in each line is the same. This is not a new idea, for example multiplicative semi--magic squares of order 3 are studied in \cite{Fried}. There are several different methods to construct multiplicative magic squares. The most obvious method is to transform an additive magic square $A=(a_{ij})$ into $M=(b^{a_{ij}})$, for any base $b$. More intriguing ways are explained in \cite{Des} (multiplicative magic squares are usually constructed by using orthogonal latin squares or geometric progressions).\\ 
In \cite{Hahn}, the problem of finding the exact number of multiplicative magic squares of order 4 has been posed. In order to solve such question, in \cite{Libis} a correspondence between additive and multiplicative magic squares is proposed. The authors claimed that such function is an isomorphism even though it is only an injective morphism, as pointed out in \cite{letter}. Since such revision is only a note, the above problem is not completely clarified and the correct number of multiplicative magic squares is given without a rigorous proof. In the next section, we clarify the above questions, proposing an original and complete proof. Moreover, in the last section a novel representation for some magic squares is proposed, allowing an easy way to construct them.\\
\section{Counting multiplicative magic squares of order 4}
A \emph{normal multiplicative magic square} $n\times n$ is created using $n$ primes $p_1,...,p_n$. All the divisors of $k=p_1\cdots p_n$ must appear once time in the matrix and in each line the product is the same. The magic constant is $k^2$ and it is well--known that in a normal multiplicative magic square, in each line, any prime compares with power one and exactly two times. For example, a normal multiplicative magic square of order 4 is
$$\begin{pmatrix} 1 & p_3 \cdot p_4 & p_1\cdot p_2 & p_1\cdot p_2 \cdot p_3 \cdot p_4 \cr p_1\cdot p_3 \cdot p_4 & p_1\cdot p_2\cdot p_3 & p_4 & p_2 \cr p_1\cdot p_2\cdot p_4 & p_1 & p_2 \cdot p_3 \cdot p_4 & p_3 \cr p_2 \cdot p_3 & p_2\cdot p_4 & p_1\cdot p_3 & p_1 \cdot p_4 \end{pmatrix}.$$
In the following, we use $\mathbb M_{k,n}$ to indicate the set of normal multiplicative magic squares $n\times n$ with magic constant $k^2$. An hard problem to solve is to find the order of $\mathbb M_{k,4}$, as posed in \cite{Hahn}. In order to answer to this question, $\mathbb M_{k,4}$ is made in correspondence with the set of \emph{normal additive magic squares} $4\times 4$ \cite{Libis}.\\ 
A \emph{normal additive magic square} is a matrix $n \times n$, whose entries consist of the numbers $0,1,...,n^2-1$ and in each line the sum of the numbers is the same. The magic constant is $\frac{n(n^2-1)}{2}$. From now on, we name $\mathbb A_n$ the set of normal additive magic squares $n\times n$.\\
It is well--known that the order of $\mathbb A_4$ is 7040, or 880 up to symmetries of the square \cite{Bondi}. Can we find a similar result for $\mathbb M_{k,4}$?\\
In \cite{Libis} the following correspondence is defined:
$$f:\mathbb M_{k,n}\rightarrow \mathbb A_{2^{n/2}}.$$
If $f$ is bijective, $\lvert \mathbb M_{k,4} \rvert$ is determined. The function $f$ maps $M=(m_{ij})\in \mathbb M_{k,n}$ into $A=(a_{ij})\in \mathbb A_{2^{n/2}}$. The element $m_{ij}$ is made in correspondence with $a_{ij}$, where $a_{ij}$ is the base 10 number , between 0 and $n^2-1$, whose base 2 expression is the $n$--string associated to $m_{ij}$. This string is constructed as follows: arrange the $n$ prime numbers in ascending order; if $p_k$ is a factor of $m_{ij}$ place a 1 in the $k$th position of the string ; if not, place a 0.\\
Such a function $f$ is surely injective, but unfortunately it is not surjective and the order of $\mathbb M_{k,4}$ is not $\lvert \mathbb A_4 \rvert=7040$. Indeed, we can consider the group $S_4$ which permutes the four primes involved in $\mathbb M_{k,4}$.
\begin{teo}
The order of $S_4$ divides the order of $\mathbb M_{k,4}$.
\end{teo}
\begin{proof}
Surely when a permutation $\rho \in S_4$ acts on $M\in\mathbb M_{k,4}$, then $\rho(M)$ is still an element in $\mathbb M_{k,4}$. Furthermore, if $\rho$ is not the identity, it does not fix any element of $\mathbb M_{k,4}$, i.e., any orbit has 24 elements. Therefore by Burnside Lemma $\lvert S_4 \rvert$ must be a divisor of $\lvert \mathbb M_{k,4} \rvert$.
\end{proof}
As immediate consequence of the previous Theorem, 7040 can not be the order of $\mathbb M_{k,4}$, since 24 does not divide 7040. The correct answer is given in the following theorem.
\begin{teo}
The order of $\mathbb M_{k,4}$ is 4224.
\end{teo}
\begin{proof}
If we consider $A\in \mathbb A_{4}$ and we write its entries in base 4, then in each line the sum of the digits in the first position multiplied by 4 and added to the sum of the digits in the second position must yield 30. The only possible combinations are
$$30=7\times4+2=6\times4+6=5\times4+10.$$
So if $A$ has a line in which the digits in the second position have sum 2, we are in the only two situations
\begin{eqnarray*}
\cdot0 \quad \cdot0 \quad \cdot1\quad \cdot1\\
 \cdot0 \quad \cdot0 \quad \cdot0 \quad \cdot2\\
\end{eqnarray*}
where $\cdot$ indicates the first position of the number in base 4. Considering the representation in base 2 of these strings, we have one of the following situations
\begin{eqnarray*}
\cdot\cdot00 \quad \cdot\cdot00 \quad \cdot\cdot01 \quad \cdot\cdot01 \\
\cdot\cdot00 \quad \cdot\cdot00 \quad \cdot\cdot00 \quad \cdot\cdot10 &.&\\
\end{eqnarray*}
Now if we try to create $M\in\mathbb M_{k,4}$, using the inverse of $f$, if $A$ has a line which presents one of these situations we can not have a normal multiplicative magic square. In fact, if we are in the first situation, for example, then the prime $p_3$ will not appear in this line. Similarly if a line of $A$ has digits in the second position with sum 10, the possible cases are
\begin{eqnarray*}
\cdot3 \quad \cdot3 \quad \cdot3 \quad \cdot1 \\
\cdot2 \quad \cdot2 \quad \cdot3 \quad \cdot3 \\
\end{eqnarray*}
which correspond in base 2 to the strings
\begin{eqnarray*}
\cdot\cdot11 \quad \cdot\cdot11 \quad \cdot\cdot11 \quad \cdot\cdot01\\
\cdot\cdot10 \quad \cdot\cdot10 \quad \cdot\cdot11 \quad \cdot\cdot11 &.& \\
\end{eqnarray*}
Once more, if we try, we fail to obtain $M\in \mathbb M_{k,4}$\ . \\
Finally, if $A$ has a line whose digits in the second position have sum 6, we have the possibilities
\begin{eqnarray*}
\cdot0 \quad \cdot1 \quad \cdot2 \quad \cdot3\\
\cdot0 \quad \cdot0 \quad \cdot3 \quad \cdot3\\
\cdot1 \quad \cdot1 \quad \cdot2 \quad \cdot2\\
\cdot0 \quad \cdot2 \quad \cdot2 \quad \cdot2\\
\cdot1 \quad \cdot1 \quad \cdot1 \quad \cdot3 &.&\\
\end{eqnarray*}
With similar arguments as the ones used before, it is easy to see that the last two situations do not allow to obtain a normal multiplicative magic square. Therefore the squares $A\in\mathbb A_4$ from which we obtain a normal multiplicative magic square are only those with lines whose entries in base 4 have digits in the first and second position composed by strings 
\begin{eqnarray*}
\cdot0 \quad \cdot1 \quad \cdot2 \quad \cdot3\\
\cdot0 \quad \cdot0 \quad \cdot3 \quad \cdot3\\
\cdot1 \quad \cdot1 \quad \cdot2 \quad \cdot2&.&\\
\end{eqnarray*}

In \cite{Bondi}, all the squares in $\mathbb A_4$ are classified. The squares in $\mathbb A_4$ generated using only these strings, i.e., which correspond to squares in $\mathbb M_{k,4}$, are only those in category one (\cite{Bondi}, p. 510). They are exactly 528 unless of symmetries of the square. Therefore our arguments and the injectivity of $f$ allow us to conclude that the order of $\mathbb M_{k,4}$ is $528\times8=4224$.
\end{proof}

\section{A new representation for magic squares}
In this section we see a representation for normal additive magic squares which correspond to normal multiplicative ones.\\ 
In \cite{Flo} a lower bound for the distance between the maximal and minimal element in a multiplicative magic square is given. In order to do that, the relation between additive and multiplicative magic squares is highlighted, recalling that a multiplicative magic square can be factorized as $\prod_i{p_i^{A_i}}$, where $A_i$'s are additive magic squares. Moreover, focusing on magic squares of order 4, in \cite{Flo} the Hilbert basis (composed by 20 magic squares) for such magic squares is explicited. However, such representation and the relation between additive and multiplicative magic squares are not really manageable in order to construct additive and multiplicative magic squares of order 4. Here, the proposed representation allows to determine all (and not only) the normal additive magic squares of order 4, which corresponds to normal multiplicative magic squares. In this way, they can all be easily constructed essentially using only 5 basic matrices.\\
We consider $M=(m_{ij})\in\mathbb M_{k,4}$, by means of $f$ we have the correspondent $A=(a_{ij})\in\mathbb A_4$ and we consider its entries in base 2. Now we decompose $A$ into four matrices $A_1, A_2, A_3, A_4$, whose entries are only 0 or 1, so that the entries in position $ij$ of the matrices $A_1$, $A_2$, $A_3$, $A_4$ form the string $a_{ij}$.
\begin{es}
From the normal multiplicative magic square 
$$\begin{pmatrix}  p_1p_2p_3 & p_3p_4 & 1 & p_1p_2p_4 \cr p_2p_4 & p_1 & p_1p_3p_4 & p_2p_3 \cr p_1p_4 & p_2 & p_2p_3p_4 & p_1p_3 \cr p_3 & p_1p_2p_3p_4 & p_1p_2 & p_4  \end{pmatrix},$$
using $f$, we obtain the normal additive magic square
$$\begin{pmatrix}  1110 & 0011 & 0000 & 1101 \cr 0101 & 1000 & 1011 & 0110 \cr 1001 & 0100 & 0111 & 1010 \cr 0010 & 1111 & 1100 & 0001  \end{pmatrix}=\begin{pmatrix}  14 & 3 & 0 & 13 \cr 5 & 8 & 11 &6 \cr 9 & 4 & 7 & 10 \cr 2 & 15 & 12 & 1  \end{pmatrix}$$
and it can be decomposed into
$$8\begin{pmatrix}  1 & 0 & 0 & 1 \cr 0 & 1 & 1 & 0 \cr 1 & 0 & 0 & 1 \cr 0 & 1 & 1 & 0  \end{pmatrix}+4\begin{pmatrix}  1 & 0 & 0 & 1 \cr 1 & 0 & 0 & 1 \cr 0 & 1 & 1 & 0 \cr 0 & 1 & 1 & 0  \end{pmatrix}+2\begin{pmatrix}  1 & 1 & 0 & 0 \cr 0 & 0 & 1 & 1 \cr 0 & 0 & 1 & 1 \cr 1 & 1 & 0 & 0  \end{pmatrix}+\begin{pmatrix}  0 & 1 & 0 & 1 \cr 1 & 0 & 1 & 0 \cr 1 & 0 & 1 & 0 \cr 0 & 1 & 0 & 1  \end{pmatrix}.$$
\end{es}
Since $A$ is derived from a matrix in $\mathbb M_{k,4}$, by means of $f$, the matrices $A_1, A_2, A_3, A_4$ are all and only those having in each line exactly two ones. We call \emph{forms} these matrices which allow us to construct any magic square in $\mathbb A_4$ corresponding to magic squares in $\mathbb M_{k,4}$. These forms are exactly 16. 
\begin{teo}
There are 16 different matrices $4\times4$, with entries 0 or 1, such that in each line there are exactly two ones.
\end{teo}
\begin{proof}
The strings that we can use to generate these forms are only six:
$$1100,\quad 1010, \quad 1001,\quad 0110,\quad 0101, \quad 0011 \ .$$
If we choose as a diagonal a string with different extremes, then we have only two possible different forms. For example:
$$\begin{pmatrix} 1 & \cdot & \cdot & \cdot \cr \cdot & 1 & \cdot & \cdot \cr \cdot & \cdot & 0 & \cdot \cr \cdot & \cdot & \cdot & 0  \end{pmatrix}\rightarrow \begin{pmatrix} 1 & \cdot & \cdot & 0 \cr \cdot & 1 & \cdot & \cdot \cr \cdot & \cdot & 0 & \cdot \cr \cdot & \cdot & \cdot & 0  \end{pmatrix}\rightarrow\begin{pmatrix} 1 & 0 & 1 & 0 \cr 0 & 1 & 0 & 1 \cr 0 & 1 & 0 & 1 \cr 1 & 0 & 1 & 0  \end{pmatrix}$$
$$\begin{pmatrix} 1 & \cdot & \cdot & \cdot \cr \cdot & 1 & \cdot & \cdot \cr \cdot & \cdot & 0 & \cdot \cr \cdot & \cdot & \cdot & 0  \end{pmatrix}\rightarrow \begin{pmatrix} 1 & \cdot & \cdot & 1 \cr \cdot & 1 & \cdot & \cdot \cr \cdot & \cdot & 0 & \cdot \cr \cdot & \cdot & \cdot & 0  \end{pmatrix}\rightarrow\begin{pmatrix} 1 & 0 & 0 & 1 \cr 0 & 1 & 1 & 0 \cr 1 & 0 & 0 & 1 \cr 0 & 1 & 1 & 0  \end{pmatrix}\ .$$
On the other hand, if we choose as a diagonal a string with same extremes, then we have four possible forms. Since we have four strings with different extremes and two strings with same extremes, our forms are $4\cdot2+2\cdot4=16\ .$ 
\end{proof}
From these 16 forms we can individuate five \emph{fundamental forms}:
$$
\begin{array}{clllll}
A_0&=&\begin{pmatrix} 0 & 0 & 1 & 1 \cr 0 & 1 & 0 & 1 \cr 1 & 0 & 1 & 0 \cr 1 & 1 & 0 & 0  \end{pmatrix} B_0&=&\begin{pmatrix} 0 & 0 & 1 & 1 \cr 1 & 1 & 0 & 0 \cr 0 & 0 & 1 & 1 \cr 1 & 1 & 0 & 0  \end{pmatrix}\\ C_0&=&\begin{pmatrix} 0 & 0 & 1 & 1 \cr 1 & 1 & 0 & 0 \cr 1 & 1 & 0 & 0 \cr 0 & 0 & 1 & 1  \end{pmatrix} D_0&=&\begin{pmatrix} 0 & 1 & 0 & 1 \cr 1 & 0 & 1 & 0 \cr 1 & 0 & 1 & 0 \cr 0 & 1 & 0 & 1  \end{pmatrix}\\ E_0&=&\begin{pmatrix} 0 & 1 & 0 & 1 \cr 1 & 1 & 0 & 0 \cr 0 & 0 & 1 & 1 \cr 1 & 0 & 1 & 0  \end{pmatrix}.\\
\end{array}
$$
All the remaining can be found acting on these five forms with the group of symmetries of the square $D_8$.
Combining any four forms, as we have done in the previous example, the resulting matrix is always an additive magic square, not necessarily \emph{normal}, and it corresponds to a multiplicative one.
\begin{teo} f $A_1$, $A_2$, $A_3$, $A_4$ are fundamental forms, or matrices obtained from fundamental forms through the action of $D_8$, then 
$$8A_1+4A_2+2A_3+A_4$$
is always an additive magic square which corresponds to a multiplicative magic square.
\end{teo}
Given a magic square $8A_1+4A_2+2A_3+A_4$ constructed using our forms, we can consider all the permutations of the forms $A_1$, $A_2$, $A_3$, $A_4$. After a permutation, we have still a magic square. Thus we can use the group $S_4$  over our forms.
\begin{es}
We consider
$$8\begin{pmatrix}  0 & 0 & 1 & 1 \cr 1 & 1 & 0 & 0 \cr 0 & 0 & 1 & 1 \cr 1 & 1 & 0 & 0 \end{pmatrix}+4\begin{pmatrix}  1 & 0 & 1 & 0 \cr 1 & 0 & 1 & 0 \cr 0 & 1 & 0 & 1 \cr 0 & 1 & 0 & 1 \end{pmatrix}+2\begin{pmatrix}  0 & 1 & 1 & 0 \cr 0 & 1 & 1 & 0 \cr 1 & 0 & 0 & 1 \cr 1 & 0 & 0 & 1 \end{pmatrix}+\begin{pmatrix}  1 & 1 & 0 & 0 \cr 0 & 0 & 1 & 1 \cr 0 & 0 & 1 & 1 \cr 1 & 1 & 0 & 0 \end{pmatrix}$$
which corresponds to
$$\begin{pmatrix}  5 & 3 & 14 & 8 \cr 12 & 10 & 7 & 1 \cr 2 & 4 & 9 & 15 \cr 11 & 13 & 0 & 6 \end{pmatrix}.$$
Permuting the positions of the forms, we obtain
$$8\begin{pmatrix}  1 & 1 & 0 & 0 \cr 0 & 0 & 1 & 1 \cr 0 & 0 & 1 & 1 \cr 1 & 1 & 0 & 0 \end{pmatrix}+4\begin{pmatrix}  0 & 0 & 1 & 1 \cr 1 & 1 & 0 & 0 \cr 0 & 0 & 1 & 1 \cr 1 & 1 & 0 & 0 \end{pmatrix}+2\begin{pmatrix}  1 & 0 & 1 & 0 \cr 1 & 0 & 1 & 0 \cr 0 & 1 & 0 & 1 \cr 0 & 1 & 0 & 1 \end{pmatrix}+\begin{pmatrix}  0 & 1 & 1 & 0 \cr 0 & 1 & 1 & 0 \cr 1 & 0 & 0 & 1 \cr 1 & 0 & 0 & 1 \end{pmatrix}$$
which is equal to
$$\begin{pmatrix}  10 & 9 & 7 & 4 \cr 6 & 5 & 11 & 8 \cr 1 & 2 & 12 & 15 \cr 13 & 14 & 0 & 3 \end{pmatrix}.$$
\end{es}
We have seen that using our forms we obtain always an additive magic square but it is not necessarily \emph{normal}. Finally, let us see how to utilize the forms in order to generate all and only the normal additive magic squares corresponding to normal multiplicative magic squares.
The orbit of the fundamental forms with respect to the action of $D_8$ are
\begin{eqnarray*}
A&=&\{A_0,A_1\}\\ 
B&=&\{B_0,B_1,B_2,B_3\}\\
C&=&\{C_0,C_1,C_2,C_3\}\\
D&=&\{D_0,D_1,D_2,D_3\}\\
E&=&\{E_0,E_1\}.
\end{eqnarray*}

We easily observe that 
$$A_0+A_1=E_0+E_1=\mathbb{U}=\begin{pmatrix} 1 & 1 & 1 & 1 \cr 1 & 1 & 1 & 1 \cr 1 & 1 & 1 & 1 \cr 1 & 1 & 1 & 1 \end{pmatrix}.$$
Furthermore, for any form in the orbits $B$, $C$, $D$ there is another form, in the same orbit, such that their sum is $\mathbb U$.

We call \emph{class} the set  $(A,B,C,D)$, whose elements are all the magic square obtained combining the forms belonging to the orbits $A$, $B$, $C$, $D$ (e.g., $8A_0+4B_1+2C_2+D_3$ or $8B_2+4C_1+2D_0+A_0$ are elements of $(A,B,C,D)$). In the next theorem we show all the classes which provide normal additive magic squares. 

\begin{oss}
\textit{We obtain normal additive magic squares, which corresponds to all the normal multiplicative magic squares, only from the classes
\begin{eqnarray*}
(A, C, D, E)\\
(B, B, C, C)\\
(B, B, C, D)\\
(B, B, D, D)\\
(B, C, C, D)\\
(B, C, D, D)\\
(C, C, D, D)&.
\end{eqnarray*}}
We have to made clear some details: 
\begin{enumerate}
\item when we choose a form in the orbit $C$, the forms available in the orbit $D$ are only two and vice versa, except for the class $(C,C,D,D)$;
\item when we have a class with two forms from the same orbits, their sum must not be $\mathbb U$;
\item we can not take two times the same form in the same class.
\end{enumerate}
Considering this remarks, we can count the magic squares obtainable from these classes and we check that they are exactly 4224.\\
For the class $(A,C,D,E)$ we can choose 2 forms from the orbit $A$, 4 from the orbit $C$, only 2 from the orbit $D$ and 2 from $E$. Thus from this class we can obtain 32 normal additive magic squares, unless of permutations. Thus we have  $32\cdot24=768$ normal additive magic squares from $(A,C,D,E)$, and we write
$$\lvert (A,C,D,E) \rvert=768.$$
Similarly we find
$$
\begin{array}{cllllll}
 \left|(B,B,C,C)\right|&=&(4\cdot2\cdot4\cdot2)& \cdot&6&=&384\\
 \left|(B,B,C,D)\right|&=&(4\cdot2\cdot4\cdot2)&\cdot&12&=&768\\
 \left|(B,B,D,D)\right|&=&(4\cdot2\cdot4\cdot2)&\cdot&6&=&384\\
 \left|(B,C,C,D)\right|&=&(4\cdot4\cdot2\cdot2)&\cdot&12&=&768\\
 \left|(B,C,D,D)\right|&=&(4\cdot2\cdot4\cdot2)&\cdot&12&=&768\\
 \left|(C,C,D,D)\right|&=&(4\cdot2\cdot4\cdot2)&\cdot&6&=&384\\
\end{array}
$$
and 
$$768 + 384 + 768 + 384 + 768 + 768 + 384=4224.$$
\end{oss}
\begin{oss}
All the magic squares that can be represented through our notation can be classified and identified by the membership class.
\end{oss}
\begin{oss}
Such representation allows to construct in a simple way all the normal multiplicative magic squares of order 4.
\end{oss}
We conclude this paper with a further example.
\begin{es}
Let us consider the factorization $2\cdot3\cdot5\cdot67$ of $2010$. We take a magic square in $\mathbb M_{2010,4}$
$$\begin{pmatrix} 2010 & 5 & 67 & 6 \cr 3 & 134 & 10 & 1005 \cr 2 & 201 & 15 & 670 \cr 335 & 30 & 402 & 1 \end{pmatrix}=\begin{pmatrix} 2\cdot3\cdot5\cdot67 & 5 & 67 & 2\cdot3 \cr 3 & 2\cdot67 & 2\cdot5 & 3\cdot5\cdot67 \cr 2 & 3\cdot67 & 3\cdot5 & 2\cdot5\cdot67 \cr 5\cdot67 & 2\cdot3\cdot5 & 2\cdot3\cdot67 & 1 \end{pmatrix}.$$
its image through $f$ is 
$$\begin{pmatrix} 1111 & 0010 & 0001 & 1100 \cr 0100 & 1001 & 1010 & 0111 \cr 1000 & 0101 & 0110 & 1011 \cr 0011 & 1110 & 1101 & 0000 \end{pmatrix}=\begin{pmatrix} 15 & 2 & 1 & 12 \cr 4 & 9 & 10 & 7 \cr 8 & 5 & 6 & 11 \cr 3 & 14 & 13 & 0 \end{pmatrix}\ .$$
This square belongs to the class $(C,C,D,D)$, in fact it can be decomposed as follows
$$8\begin{pmatrix} 1 & 0 & 0 & 1 \cr 0 & 1 & 1 & 0 \cr 1 & 0 & 0 & 1 \cr 0 & 1 & 1 & 0 \end{pmatrix}+4\begin{pmatrix} 1 & 0 & 0 & 1 \cr 1 & 0 & 0 & 1 \cr 0 & 1 & 1 & 0 \cr 0 & 1 & 1 & 0 \end{pmatrix}+2\begin{pmatrix} 1 & 1 & 0 & 0 \cr 0 & 0 & 1 & 1 \cr 0 & 0 & 1 & 1 \cr 1 & 1 & 0 & 0 \end{pmatrix}+\begin{pmatrix} 1 & 0 & 1 & 0 \cr 0 & 1 & 0 & 1 \cr 0 & 1 & 0 & 1 \cr 1 & 0 & 1 & 0 \end{pmatrix}.$$
\end{es}


\begin{thebibliography}{20}

\bibitem{Andrews} W. S. Andrews, Magic squares and cubes, Dover Publications, New York, 1960.

\bibitem{Bona} M. Bona, A new proof of the formula for the number of $3\times3$ magic squares, \textit{Mathematics Magazine}, Vol. \textbf{70} (1997), 201--203.

\bibitem{Bona2} M. Bona, Introduction to enumerative combinatorics, McGraw Hill, 2007.

\bibitem{Bondi} Sir H. Bondi, Dame K. Ollerenshaw, Magic squares of order four, \textit{Phil. Trans. R. Soc. Lond.} A, \textbf{306} (1982), 443--532.

\bibitem{Flo} J. Cilleruelo, F. Luca, On multiplicative magic squares, \textit{The Electronic Journal of Combinatorics}, Vol. \textbf{17}, $\#$N8, 2010. 

\bibitem{Des} R. Descombes, Les carrés magiques, Vuibert, 2000.

\bibitem{Fried} D. Friedman, Multiplicative magic squares, \textit{Mathematics Magazine}, Vol. \textbf{49}, No. \textbf{5}, (1976), 249--250.

\bibitem{Hahn} Macalester College Problem of the Week, Nov. 30, 1994.

\bibitem{Libis} C. Libis, J. D. Phillips, M. Spall, How many magic squares are there?, \textit{Mathematics Magazine}, Vol. \textbf{73} (2000),  57--58.

\bibitem{letter} L. Sallows, C. Libis, J. D. Phillips, S. Golomb, \emph{News and letters}, Mathematics Magazine, Vol. \textbf{73}, No. \textbf{4}, 332--334, 2000.

\end{thebibliography}
\end{document}